\documentclass[a4paper,reqno]{amsart}
\usepackage{graphicx,amsmath,amssymb}
\usepackage{algorithm}
\usepackage{mathdots, comment}
\usepackage[noend]{algpseudocode}
\usepackage{amsmath,stackengine}
\stackMath
\makeatletter
\def\BState{\State\hskip-\ALG@thistlm}
\makeatother
\allowdisplaybreaks

\theoremstyle{plain}
\newtheorem{theorem}{Theorem}[section]
\newtheorem{proposition}[theorem]{Proposition}
\newtheorem{lemma}[theorem]{Lemma}
\newtheorem{corollary}[theorem]{Corollary}
\newtheorem{conjecture}[theorem]{Conjecture}
\theoremstyle{definition}
\newtheorem{definition}[theorem]{Definition}
\newtheorem{example}[theorem]{Example}
\newtheorem{remark}[theorem]{Remark}
\newtheorem{question}[theorem]{Question}
\theoremstyle{observation}
\newtheorem{observation}[theorem]{Observation}

\begin{document}
\title{conditions for matchability in groups and field extensions}

\author[M. Aliabadi, J. Kinseth, C. Kunz, H. Serdarevic, C. Willis]{Mohsen Aliabadi, Jack Kinseth, Christopher Kunz, Haris Serdarevic, Cole Willis}
\address{Mohsen Aliabadi \\ 
Department of Mathematics\\ Iowa State University}%
\email{aliabadi@iastate.edu}
\address{Jack Kinseth \\
Department of Mathematics\\ Iowa State University}
\email{jkinseth@iastate.edu}
\address{Christopher Kunz \\
Department of Mathematics\\ Iowa State University}
\email{ckunz@iastate.edu}
\address{Haris Serdarevic \\
Department of Mathematics\\ Iowa State University}
\email{serharis@iastate.edu}
\address{Cole Willis  \\
Department of Mathematics\\ Iowa State University}
\email{cswillis@iastate.edu}


\begin{abstract}
 The origins of the notion of matchings in groups spawn from a linear algebra problem proposed by E. K. Wakeford \cite{Wakeford} which was tackled in 1996 \cite{Fan}. In this paper, we first discuss unmatchable subsets in abelian groups. Then we formulate and prove linear analogues of results concerning matchings, along with a conjecture that, if true, would extend the primitive subspace
theorem. We discuss the dimension m-intersection property for vector spaces and its connection to matching subspaces in a field extension, and we prove the linear version of an intersection property result of certain subsets of a given set.
\end{abstract}

\maketitle

\section{Introduction}
Throughout this paper, we may assume that $G$ is an additive abelian group, unless stated otherwise.
Let $B$ be a finite subset of  $G$ which does not contain the neutral element. For any subset $A$ in $G$ with the same cardinality as $B$, a {\it matching} from $A$ to $B$ is defined to be a bijection $f:A\to B$ such that for any $a\in A$, we have $a+f(a)\not\in A$. Evidently, it is necessary for the existence of a matching from $A$ to $B$ that
$\#A=\# B$ and $0\not\in B$. One says that a group $G$
has the {\it matching property} if these necessary conditions are sufficient as well. That is, we say that $G$ has the matching property if for any pairs of finite subsets $A$ and $B$ of it, the conditions $\#A=\#B$ and $0\notin B$ suffice to guarantee the existence of a matching between $A$ and $B$. The notion of matchings in abelian groups was introduced by Fan and Losonczy in \cite{Fan} in order to generalize a geometric property of lattices in Euclidean space related to an old problem of E. K. Wakeford concerning canonical forms for symmetric tensors. In particular, Wakeford in \cite{Wakeford}  considered  the  question of which sets of monomials are removable from a generic  homogeneous polynomial through a linear change in its variables.

The notion of matching has been investigated in literature extensively  in various ways. See \cite{Alon, Eliahou 1, Aliabadi 1} for more results on matchings. A related notion is that of a matching between subspaces of a field extension. In \cite{Eliahou 3}, Eliahou and Lecouvey formulate some linear analogues of matchings in groups and prove similar results in the linear context.  Later, the linear version of a matching is extensively studied by the first author with collaborators in \cite{Aliabadi 1, Aliabadi 2, Aliabadi 3,  Aliabadi 4}. There are still many fascinating open problems in this
area. This paper continues to study some problems motivated in \cite{Aliabadi 1, Aliabadi 2,  Aliabadi 4}. We extend our results on matchings
in groups to the linear setting, which generalizes some results of \cite{Aliabadi 2}. We study primitive subspaces and their applications in partitioning finite fields. Finally, in a related result to matchings, we discuss the dimension $m$-intersection property for vector subspaces. The analogy between matchings in abelian groups and in
field extensions is highlighted throughout the paper, and numerous open
questions are presented for further inquiry. Our tools mix linear algebra and combinatorial number theory.

\subsection{Main results}
We state our main theorems. The needed definitions from matchings in groups and linear matchings appear in Sections \ref{Matching property} and \ref{A Dimension Criterion}. We start with the following theorem in which we present the size of the largest matchable subsets of two given sets.
\begin{theorem}\label{Unmatchable}
Let $G$ be an abelian group and $A$ and $B$ be nonempty finite subsets of $G$ with  $\#A=\#B$ and $A+B\neq A$. Assume that $A$ is not matched  to $B$. Then  $M(A,B)=\#A-D(A,B)$.
\end{theorem}
The following theorem is concerned with the primitive subspace theorem which is motivated by certain matchable vector subspaces of a simple field extension.
\begin{theorem}\label{general field}
Let $\mathcal{F}$ be a nonempty finite collection of proper subspaces of a $K$-vector space $V$, where $\dim_K(V)=n<\infty$. Let $\#\mathcal{F}\leq \# K$ and $s=\max\{\dim_K(S)\mid S\in\mathcal{F}\}$. Let $T\subseteq V$ be a subspace maximal with the property that $T\cap S=0$ for all $S\in\mathcal{F}$. Then $\dim_K(T)=n-s$.
\end{theorem}
In the following theorem, we present the linear analogue of a theorem pertaining to matchings in the group setting. 
\begin{theorem}\label{Linear analog of matchings}
Let $K\subset F$ be a field extension, $A$ and $B$ be two $n$-dimensional $K$-subspaces of $F$, and $n\geq1$. Assume further that for any $b\in  B\setminus\{0\}$, $A$ does not contain any nontrivial linear translate of $K(b)$. Then $A$ is matched to $B$. 
\end{theorem}

\begin{theorem}\label{vector space span}
Let $K\subset F$ be a field extension and $A$ and $B$ be $n$-dimensional $K$-subspaces of $F$ with $n>1$. If $A$ is matched to $B$, then  $\langle AB \rangle\neq A$.
\end{theorem}
The following linear algebra result relies on two theorems. The first one is a theorem due to Rado from \cite{Rado} in which the necessary and sufficient condition for existing a free transversal is provided. The second one is an observation concerning a property for vector subspaces called the $m$-intersection property.
\begin{theorem}\label{linear algebra}
Let $W$ be an $n$-dimensional vector space and $ \mathcal{U}=\{U_1,U_2,\ldots,U_t\}$, $t<n$, be a family of subspaces of $W$ each of dimension $n-1$,  and assume that $\mathcal{U}$ satisfies the dimension $n$-intersection property. Then there exist $n-t$ subspaces $U_{t+1},\ldots,U_n$ of $W$ of dimension $n-1$ and a basis $\{x_1,\ldots,x_n\}$ for $W$ such that 
\begin{align*}
 \ker x_i^*+\ker x_j^*=W,
\end{align*}
for any $i$ and $j$ with $U_i\neq U_j$, $1\leq i,j\leq n$. 
\end{theorem}

We now present an outline of the paper. In Section \ref{Matching property}, we discuss matchings in the context of abelian groups and connect this notion to matchings in bipartite graphs. With this, along with a result on maximum matchings in bipartite graphs, we elucidate the algebraic structure of unmatchable subsets. In Section \ref{A Dimension Criterion}, we present a generalization of a linear algebra result on primitive subspaces of field extensions which arose from matching subspaces in simple field extensions. In Section \ref{The Linear Matching Property}, we formulate and prove linear analogues of results concerning matchings in groups. Section \ref{intersection property} establishes a link between matchable subspaces and a certain property of finite families of vector subspaces called the dimension $m$-intersection property. Finally, in Section \ref{Future}, we present a possible direction for future work in this line of research.

\section{Matching Property in Abelian Groups}\label{Matching property}
To begin our investigation, we note that many results on the problem of
classifying matchable subsets in groups 
are known. One of the earliest results in this direction appears in \cite{Losonczy}, where it is shown that an abelian group satisfies the
matching property if and only if it is either torsion-free or of prime order. Later, this result is generalized for arbitrary groups \cite{Eliahou 1}.  This classification was established using methods pertaining
to additive number theory and combinatorics. Specifically, the additive tools
used are lower bounds on the size of the sumset
\[
A+B=\{a+b:\ a\in A\,    \text{and}\   b\in B\},
\]
in $G$, and the main combinatorial tool is a result due to Philip Hall \cite{Hall} which states as follows:
\begin{theorem}[Hall's marriage theorem]
Let $\mathcal {G}_{A,B}=(V(\mathcal {G}_{A,B}),E(\mathcal {G}_{A,B}))$ be a bipartite graph with bipartitions $A$ and $B$ so that $\#A=\#B$. Then $\mathcal {G}$ has a perfect matching if and only if for each subset $S$ of $A$, $\#S \leq \#N(S)$, where $N(S)$ denotes the set of
vertices which are adjacent to at least one vertex in $S$

\end{theorem}

Having the classification of groups in terms of the matching property in place, a natural question one might raise is: given an arbitrary group $G$, is there any criterion to characterize matchable subsets in a more general way? This problem is studied in \cite{Aliabadi 2, Aliabadi 3} which highlights a close relation between matchable subsets and certain cosets of $G$. In particular, it is observed in \cite{Eliahou 1} that the 
existence of nontrivial proper finite subgroups is an obstruction for the matching property. Inspired by this observation, the following is proved in \cite {Aliabadi 2}.

\begin{proposition}\label{Matching generalziation}
Let $G$ be an abelian group and let $A$ and $B$ be finite subsets of G with the
same cardinality. Assume further that for any element $b \in B$, A does not contain any coset
of the subgroup generated by $b$. Then there is a matching from $A$ to $B$.
\end{proposition}

Motivated by Proposition \ref{Matching generalziation}, one may ask that if $A$ is matched to $B$, can we conclude that for any element $b\in B$, $A$ does not contain any subgroup generated by $b$? The answer is negative. For example, consider $G=\mathbb{Z}/6\mathbb{Z}$ and $A=B=G\setminus\{\bar{0}\}$. Then  $A$ is matched to $B$ via the map $\bar{a}\mapsto -\bar{a}$, but $A$ contains $\{\bar{1},\bar{4}\}=\bar{1}+\langle\bar{3}\rangle$, a coset of the subgroup generated by $\bar{3}\in B$.

 In Section \ref{The Linear Matching Property}, we will formulate and prove a linear analogue of Proposition \ref{Matching generalziation}.

\subsection{Unmatchable subsets}
 All preceding results in the literature on matchings in the group setting address the conditions and cases in which certain subsets are matchable. In this subsection, we briefly investigate unmatchable subsets. Let $A$ and $B$ be two finite nonempty subsets of an abelian group $G$ with the same cardinality $n$ and $0\not\in\ B$. Assume further that $A$ is not matched to $B$. We are interested in determining the size of the largest possible subset $A_0$ of $A$ for which $A_0$ can be matched to a subset of $B$ in the usual sense. Denote the size of such a maximum subset by $M(A,B)$ provided that $A+B\neq A$, and $M(A,B)=0$ if and only if $A+B=A$. Motivated by this definition, we investigate the structure of subsets with $A+B=A$.
 
\begin{lemma}\label{Coset}
Let $A$ and $B$ be nonempty finite subsets of an arbitrary group $G$. Assume that $\# A\leq\# B$ and that $A+B=A$. Then $B$ is a subgroup of $G$, and $A$ is a coset of $B$.
\end{lemma}
\begin{proof}
If $b$ is in $B$, then the mapping $\underset{a \longmapsto a+b}{\varphi:   A \rightarrow A+b}$ is injective, and thus $\#(A+b)=\# A$. Since $A+b\subset  A+B=A$ and $A$ is finite, it follows that $A+b=A$. Now let  $X=\{x\in G: A+x=A\}$. Then $B\subset  X$ and $X$ is a subgroup of $G$. Also $A+X=A$. If $a\in A$, then $a+X$ is a left coset of the subgroup $X$, and thus $\#(a+X)=\# X$, and we have $a+X\subset  A+X=A$. Then $\# B\geq \# A\geq \#(a+X)=\# X\geq \# B$, so $\# X=\# B$, and $\# A=\#(a+X)$. Since $B$ is finite and contained in $X$, and $\#B=\# X$, it follows that $B=X$, so $B$ is a subgroup as desired. 

Since $A$ is finite and contains $a+X$, and $\# A=\# (a+X)$, it follows that $A=a+X$. We know that $X=B$, so $A=a+X=a+B$, and thus $A$ is the left coset $a+B$ of $B$. The proof is complete.
\end{proof}

\begin{corollary}\label{Lemma consequence}
Let $A$, $B$ and $G$ be as in Lemma \ref{Coset}. Then $0\in B$.
\end{corollary}
\begin{proof}
It is immediate from $B$ being  a subgroup of $G$. 
\end{proof}

\begin{corollary} 
Let $A$, $B$ be nonempty finite subsets of an arbitrary group $G$ of the same cardinality. Then $M(A,B)$=0 if and only if $B$ is a subgroup of $G$ and $A$ is a left coset of $B$.
\end{corollary}
\begin{proof}
It is immediate.
\end{proof}
We associate a bipartite graph $\mathcal {G}_{A,B}=(V(\mathcal {G}_{A,B}),E(\mathcal {G}_{A,B}))$ to the pair of sets $A$ and $B$ as follows. The nodes of $\mathcal {G}_{A,B}$ are given by the bipartition $V(\mathcal {G}_{A,B})=A\cup B$, and there is an edge $e(a,b)\in E(\mathcal {G}_{A,B})$ joining $a\in A$ to $b\in B$ if and only if $a+b\notin A$.

Assuming $A$ is not matched to $B$, Hall's condition fails for some $S\subset\ A$, i.e. $\#S>\#N(S)$, where $N(S)$ stands for the set of vertices which are adjacent to at least one vertex in $S$. Define $D(A,B)= \text{max}\{\#S-\# \cup_{a\in S}\ B_a: S\subset A\}$. Since $A$ is not matched to $B$, then $D(A,B)> 0$. In what follows, we prove Theorem \ref{Unmatchable} in which the size of the largest matchable subsets of $A$ and $B$ is provided.
Our approach to proving the statement requires that we adapt the existing proofs for Hall's marriage theorem. In other words, we shall employ an argument similar to that of the existence of perfect matchings in balanced bipartite graphs.

\begin{proof}[Proof of Theorem \ref{Unmatchable}.]
Let $D(A,B)=d$ and $\# A=n$. It suffices to show that the bipartite graph $\mathcal {G}$ associated to $(A,B)$ has a matching of size $\#A-d$ and that the size of every matching in $\mathcal {G}$ is less than or equal to $\# A-d$. We break the proof down into two steps:

\noindent
{\textbf{Step 1:}} $M(A,B)\leq n-d$:  According to the definition of $D(A,B)$, at least $d$ vertices of $A$ will remain unmatched in any matching in $\mathcal {G}$ in the graph setting. This implies $M(A,B)\leq d-n$.

\noindent
{\textbf{Step 2:}} $M(A,B)\geq n-d$: Suppose $M(A,B)=n-k$, for some $0\leq k \leq {n-1}$. Then in our matching, there must be $k$ unmatched vertices in $A$ for which the
alternating tree rooted at these vertices does not contain an augmenting path. We can construct a set $S$ with $\#S-\#\cup_{a\in S}\ B_a=k$, where $B_a$ is defined as $B_a=\{y\in B: a+y\notin A\}$. Since $D(A,B)$ is the maximum of such differences, it follows that $\# D(A,B)\geq k$. This implies that  $n-\# D(A,B)\leq n-k$. So $M(A,B)\geq n-d$.\

By step 1 along with step 2 we totally arrive at the desired result.
\end{proof}
\begin{remark}
 Note that the method of associating a bipartite graph to our subsets in Theorem \ref{Unmatchable} first was used in \cite{Aliabadi 1} as a tool to count the number of matchings of matchable subsets of a given abelian group. See also \cite{Hamidoune} for more details about the counting aspect of matchable pairs of sets.
\end{remark}
\section{A Dimension Criterion for Primitive Matchable Subspaces}\label{A Dimension Criterion}
In this section, we shall assume that $K\subset F$ is a field extension, $A,B\subset F$ are two $n$-dimensional $K$-subspaces of $F$, and $\mathcal{A}=\{a_1,\ldots,a_n\}$, $\mathcal{B}=\{b_1,\ldots,b_n\}$ are ordered bases of $A,B$, respectively. The Minkowski product $AB$ of $A$ and $B$ is defined as $AB:=\{ab:\, a\in A, b\in B\}$. Note that Eliahou and Lecouvey have introduced the following notions for matchable bases of subspaces in a field extension \cite{Eliahou 3}. The ordered basis $\mathcal{A}$ is said to be {\it matched} to an ordered basis $\mathcal{B}$ of $B$ if 
\begin{align*}
 a^{-1}_iA\cap B\subset \langle b_1,\ldots,\hat{b}_i,\ldots,b_n\rangle,
\end{align*}
for each $1\leq i\leq n$, where $\langle b_1,\ldots,\hat{b}_i,\ldots,b_n\rangle$ is the vector space spanned  by  $\mathcal{B}\setminus\{b_i\}$. The subspace $A$ is {\it matched}  to the subspace $B$ if every basis of $A$ can be matched to a basis of $B$. A {\it strong matching} from $A$ to $B$ is a linear transformation $T:A\to B$ such that every basis $\mathcal{A}$ of $A$ is matched to the basis $T(\mathcal{A})$ of $B$. Finally, the extension $F$ of $K$ has the {\it linear matching property} if for every pair $A$ and $B$
of $n$-dimensional $K$-subspaces of $F$ with $n>1$ and $1\not\in B$, $A$ is matched to $B$.

It is shown in \cite{Losonczy} that for a nontrivial finite cyclic group $G$ and finite nonempty subsets $A$, $B$ of $G$  with $\# A=\# B$, there exists a matching from $A$ to $B$ if every element of $B$ is a generator of $G$. The linear analogue of this result is given in \cite{Aliabadi 3} as the following theorem.

\begin{theorem}\label{primitive matchings}
Let $K\subset F$ be a separable field extension and $A$ and $B$ be two $n$-dimensional $K$-subspaces of $F$ with $n>1$. Then $A$ is matched to $B$ provided that $B$ is a primitive $K$-subspace of $F$. 
\end{theorem}

Note that a $K$-subspace $B$ of $F$ is called {\it primitive} if $K(\alpha)=F$, for all $\alpha\in B\setminus\{0\}$.

\begin{remark}
It is worth pointing out that if $B$ is a primitive $K$-subspace of $F$, then $K\cap B=\{0\}$ (Here we are assuming that $K\subsetneq F$).
\end{remark}
\begin{example}
Consider the field extension $\mathbb{R}\subset\mathbb{C}$ and the $\mathbb{R}$-subspace $W=\langle i\rangle$ of $\mathbb{C}$, where $\langle i\rangle$ stands for the $\mathbb{R}$-subspace of $\mathbb{C}$ generated by $i$. Then adjoining any nonzero element of $W$ to $\mathbb{R}$ covers the entirety of $\mathbb{C}$. So $W$ is a primitive $\mathbb{R}$-subspace of $\mathbb{C}$.
\end{example}
Motivated by Theorem \ref{primitive matchings} one may ask the size of primitive subspaces. This topic is studied in \cite{Aliabadi 4, Aliabadi 1}. It is proved in  \cite{Aliabadi 4} that if $A$ is a primitive $K$-subspace of $F$ where $K$ is an infinite field, then $\dim_K A\leq [F:K]-\psi(F,K)$,  where
\[\psi(F,K)=\max\left\{[M:K]:\, M\text{ is a proper intermediate field of }K\subset F\right\}.\]
Note that in the above definition, "proper intermediate field of $K\subset F$" means $K\subset  M\subsetneq F$. Hence we have $1\leq\psi(F,K)<[F:K]$.

In particular, the dimension of the largest primitive subspace is given in \cite{Aliabadi 4} in the case the base field is infinite. Later in  \cite{Aliabadi 1}, this result is generalized for all base fields as follows:

\begin{proposition}\label{th2.1}
Let $F, K, n$ and $\psi(F,K)$ be as above. Assume that $K$ is infinite and  $K\subset F$ is simple.  Then 
\begin{align*}
\psi(F,K)+\phi(F,K)=n,
\end{align*}
where
\[\phi(F,K)=\max\left\{\dim_KV:\, V\text{ is a primitive }\text{K-subspace of F}\right\},\]
namely, $\phi(F,K)$ denotes the dimension of the largest primitive subspace.
\end{proposition}

\begin{example}\label{er}
Consider the finite field extension $\mathbb{Q}\subset \mathbb{Q} (\sqrt{2}, \sqrt{3})$. Then according to Proposition \ref{th2.1}, the dimension of the greatest primitive $\mathbb{Q}$-subspace of $\mathbb{Q} (\sqrt{2}, \sqrt{3})$ is $2$ as $[\mathbb{Q} (\sqrt{2}, \sqrt{3}):\mathbb{Q}]=4$. Thus, $\psi (\mathbb{Q} (\sqrt{2}, \sqrt{3}),\mathbb{Q})=2$.
\end{example}
In Theorem \ref{general field}, we generalize Proposition \ref{th2.1}. Note that the main required tools in the proof of Theorem \ref{general field} are linear covering results vector spaces stated as Lemma \ref{Covering} and Lemma \ref{Main Lemma union} in the next subsection.
\subsection{Linear covering results}
We begin with a well-known linear algebra theorem which asserts that a
vector space over an infinite field cannot be written as a finite union of its
proper subspaces. One can see \cite{Friedland-Aliabadi, Roman} for more details; however, in the case that the base field is finite, this result does not hold. We have the following scenario for the finite base field: let $V$ be a finite-dimensional vector space over $\mathbb{F}_{q}$, where $\mathbb{F}_{q}$ stands for finite field of order $q$, where  $q=p^r$ for some prime $p$ and  $r\in\mathbb{N}$. We say  a collection $\{W_i\}_{i\in I}$ of proper $K$-subspaces of $V$ is  a \textit{linear covering} of $V$ if $V=\underset{i\in I}{\bigcup} W_i$. 
 The \textit{linear covering number} $\mathrm{LC}(V)$ of a vector space $V$ of dimension at least $2$ is the least cardinality $\#I$ of a linear covering $\{W_{i}\}_{i \in I}$ of $V$. Under the condition $\dim_KV\geq2$, which is the sufficient and necessary condition for the existence of linear coverings, we have the following result from \cite{Heden}. See also \cite{Javaheri, Khare, Luh} for more developments on the topic of covering vector spaces.

\begin{lemma}\label{Covering}
If $\dim_KV$ and $\#K$ are not  both infinite, then  $\mathrm{LC}(V)=\#K+1$.%
\end{lemma}

Having the covering theorem for infinite base fields along with Lemma \ref{Covering} at hand, we obtain the following covering result for  arbitrary base fields.

\begin{lemma}\label{Main Lemma union}
Let $V$ be a finite-dimensional vector space over a field $K$ and let $\mathcal{V}=\{V_i\}_{i=1}^{m}$ be a finite family of subspaces of $V$ where $m\leq\#K$. Then $V\neq\cup_{i=1}^{m}V_{i}$.
\end{lemma}
\begin{proof}
In case $K$ is finite, the proof is an immediate consequence of Lemma \ref{Covering}. If $K$ is infinite, the proof follows from Theorem 1.2 in \cite{Roman}.
\end{proof}

The following short lemma will be used in the proof of Theorem \ref{general field}:

\begin{lemma}\label{Sum of subspaces}
Let $A$, $B$ and $C$ be subspaces of a vector space $V$, and suppose $A\cap B=0$ and $(A+B)\cap C=0$. Then $(A+C)\cap B=0$.
\end{lemma}
\begin{proof}
Assume to the contrary that $(A+C)\cap B\neq0$. Let $b\in (A+C)\cap B$, where $b\neq0$, and write $b=a+c$, with $a\in A$ and $c\in C$.

Then $c=b-a$ lies in $C\cap (A+B)=0$, so $c=0$. Thus $a=b$ lies in $A\cap B=0$. Therefore $a=0=b$. This contradicts the fact that $b\neq0$.
\end{proof}

In the proof of Theorem \ref{general field}, we assume that $dim_K{V}=n>2$. The cases $n=1$ and $n=2$ are straightforward to verify. 

\begin{proof}[Proof of Theorem \ref{general field}]
 Let $t=\dim_K(T)$. If $S\in\mathcal{F}$, then $T\cap S=0$, so 
\[n=\dim_KV\geq \dim_K(T+S)=\dim_K(T)+\dim_K(S)=t+\dim_K(S),\]
and thus $\dim_K(S)\leq n-t$ for all $S\in\mathcal{F}$. Thus $s\leq n-t$, and so $t\leq n-s$.

To complete the proof, we show that $t=n-s$. Otherwise, $t<n-s$, so $n>s+t$. Then
\[\dim_K(V)=n>s+t\geq\dim_K(S)+\dim_K(T)\geq\dim_K(S+T),\]
for all $S\in\mathcal{F}$. Then $S+T$ is a proper subspace of $V$ for all $S\in\mathcal{F}$. 

By Lemma \ref{Main Lemma union}, there exists a vector $v\in V$ such that $v\notin S+T$, and thus $(S+T)\cap Kv=0$ for all $S\in\mathcal{F}$. Also, by the definition of $T$, we have $S\cap T=0$, and thus by Lemma \ref{Sum of subspaces}, we have $(Kv+T)\cap S=0$ for all $S\in\mathcal{F}$. Now $v\notin S+T$, so $v\notin T$, and thus $T\subsetneq Kv+T$. This contradicts the definition of $T$.
\end{proof}

Observe that the condition $m\leq\#K$ in Lemma \ref{Main Lemma union} also appears in Theorem \ref{general field} as the covering theorem for vector spaces over finite fields plays a crucial role in the proofs of Theorem \ref{general field}. However, we do not encounter such a restriction when the base field is infinite. Inspired by this observation, to determine whether or not the condition $\# \mathcal{F}\leq\#K$ is removable from Theorem \ref{general field}, one may take finite-dimensional vector spaces over finite fields into account. In what follows, we first determine by an example that in Lemma Theorem \ref{general field} the condition $m\leq\#K$ cannot be relaxed. Our example signifies that the upper bound $\#K$ for the number of subspaces is strict.

\begin{example}
Consider the vector space $\mathbb{F}_{2}^{2}$ over $\mathbb{F}_{2}$ and $\mathbb{F}_{2}$-subspaces\linebreak $V_1$=$\{(0,0), (1,0)\}$, $V_2$=$\{(0,0), (0,1)\}$ and $V_3$=$\{(0,0), (1,1)\}$. The family $\mathcal{V}=\{V_i\}_{i=1}^{3}$ violates the condition $\# \mathcal{F}\leq\#K$ in Theorem \ref{general field}. The largest possible dimension of a subspace $W$ of $\mathbb{F}_{2}^{2}$ which intersects every $V_i$ trivially is zero. Hence, Theorem \ref{general field} fails for $\mathcal{V}$.
\end{example}

\begin{question}\label{Conjecture}
Let $V$ be an $n$-dimensional vector space over a field $K$ and let $\mathcal{V}=\{V_i\}_{i<n,i\mid n}$ be a finite family of subspaces of $V$ indexed by
positive proper divisors of $n$ that satisfy the following two properties: 

i) $dim_KV_i=i,$

ii) $V_i\cap V_j=V_\text{gcd(j,j)}$.

Then is it true that the dimension of the largest possible subspace of $V$ which intersects every member of $\mathcal{V}$  trivially is given by the following?
\[\dim_KV-\text{the largest divisor of}\ n.\]
\end{question}

 \begin{remark}
Along the same line of reasoning as in the proof of Theorem \ref{general field}, one may prove the question in the case of an infinite field by invoking the fact that a vector  space  over  an  infinite field cannot be written as a finite union of  its proper subspaces. Therefore, everything boils down to the case where the base field is finite. We believe that in order to handle this case, we require stronger tools than covering results for vector spaces over finite fields.
\end{remark}

\subsection{A connection to a group theory result}
There has been a vast literature as well as ongoing investigations on linear
analogues of existing results in group theory. As a case  in point, a recent result  due to Bachoc et al \cite{Bachoc}, gives the linearization of a theorem of Kneser on the size of certain subsets of an abelian group. We consider the following scenario in group theory.

Let $\mathbb{Z}/p^r\mathbb{Z}$ denote the cyclic group of order $p^r$, where $p$ is a prime and $r\in \mathbb{N}$. Denote the order of greatest proper subgroup of $\mathbb{Z}/p^r\mathbb{Z}$ by $\psi(\mathbb{Z}/p^r\mathbb{Z})$, and denote the number of generators of $\mathbb{Z}/p^r\mathbb{Z}$ by $\phi(\mathbb{Z}/p^r\mathbb{Z})$. Since $p$-groups have subgroups of index $p$, then $\psi(\mathbb{Z}/p^r\mathbb{Z})=p^{r-1}$. Also, it is well known that $\phi(\mathbb{Z}/p^r\mathbb{Z})=\varphi(p^r)$, where $\varphi$ stands for Euler's totient function. According to Euler's product formula, $\varphi(p^r)=p^r\left(1-\frac{1}{p}\right)=p^r-p^{r-1}$. Therefore,
\begin{align} \label{primitive subspace theorem}
\psi(\mathbb{Z}/p^r\mathbb{Z})+\phi(\mathbb{Z}/p^r\mathbb{Z})=p^r.
\end{align}
Note that Proposition \ref{th2.1} can be regarded as a linear analogue of relation \eqref{primitive subspace theorem}. 

Indeed, the linear analogues of ``the  order  of a group'', ``the order of its largest proper subgroup'' and ``the number of generators of a cyclic group'' are ``the degree of a field extension'', ``the degree of its largest proper intermediate subfield'' and ``the dimension of the largest primitive vector space'', respectively.

The linear analogue of the group theory result presented
as Proposition \ref{th2.1} (for general base fields)  seems to be a better result than the original group theory theorem
which only applies to cyclic groups of ``prime power'' order, as it applies to all finite degree extensions $F/K$ that are 	``cyclic'' (i.e., monogenic as a $K$-algebra). Formulating a reasonable analogue for all finite cyclic groups shall be possible in some ways. 

\subsection{Partitioning finite fields}
Consider the field extension $\mathbb{F}_q\subset\mathbb{F}_{q^n}$, where $n\in \mathbb{N}$ and $q=p^{r}$ for some prime $p$ and $r \in \mathbb{N}$. Let $V$ be an $\mathbb{F}_q$-subspace of $\mathbb{F}_{q^n}$. We call a set $\mathcal{P}=\{W_i\}_{i=1}^\ell$ of $\mathbb{F}_q$-subspaces of $\mathbb{F}_{q^n}$ a \textit{partition} of $V$ if every nonzero element of $V$ is in $W_i$ for exactly one $i$. See \cite{Heden} for more results on partitions of finite vector spaces.

In the following observation, we provide a partition for $\mathbb{F}_{q^n}$ using its primitive $\mathbb{F}_q$-subspaces.

\begin{observation}\label{partition}
Consider the field extension $\mathbb{F}_q\subset\mathbb{F}_{q^n}$. Let $M$ be an intermediate subfield of $\mathbb{F}_q\subset\mathbb{F}_{q^n}$ for which $\psi(\mathbb{F}_{q^n}, \mathbb{F}_q)=[M:K]$. Let $W$ be a $\mathbb{F}_q$-primitive subspace of $\mathbb{F}_{q^n}$ such that $\phi(\mathbb{F}_{q^n}, \mathbb{F}_q)=\dim_{\mathbb{F}_q}W$. Assume that $W$ has a subspace partition $\{W_1,\ldots,W_l\}$, where $\dim_{\mathbb{F}_q}W_i=t_i\leq\psi(\mathbb{F}_{q^n},\mathbb{F}_q)$, for $1\leq i\leq l$. Then, for each $\alpha\in M\setminus\{0\}$, one can define a $t_i$-dimensional subspace $W_{i_\alpha}$ of $\mathbb{F}_{q^n}$ such that $W$, $M$ and the subspaces $W_{i_\alpha}$ form a partition of $\mathbb{F}_{q^n}$.
\end{observation}
\begin{proof}
For each subspace $W_i$, $1\leq i\leq l$, let $T_i$ be a 1-1 linear transformation from $W_i$ into $M$. For each $\alpha\in M\setminus\{0\}$, we associate with it  the following set:
\begin{align*}
W_{i_\alpha}=\{w+\alpha T_i(w):\ w\in W_i\}.
\end{align*}
Since $W_i\cap M=\{0\}$ and $W_i\cap W_j=\{0\}$, for $i\neq j$, one can easily verify that the $W_{i_\alpha}$'s, $M$ and $W$ form a partition of $\mathbb{F}_{q^n}$ into  subspaces.
\end{proof}

\section{The Linear Matching Property, Improved}\label{The Linear Matching Property}

Our main goal in this section is to formulate and prove a linear analogue of Proposition \ref{Matching generalziation}.
  For this purpose, we employ  the following result from \cite{Bachoc} which  is the linear version of a famous theorem due to Kneser \cite[page 116, Theorem 4.3]{Nathanson}. Note that in the following theorem, $\langle AB\rangle$ stands for the $K$-subspace of $F$ spanned by the subset \[
AB=\{ab:\ a\in A\, \text{and}\, b\in B\}
,\]
which is the Minkowski product of the subspaces $A$ and $B$.

\begin{proposition}\label{subpaces product}
Let $K\subset F$ be a field extension, and let $A,B\subset F$ be nonzero finite-dimensional $K$-subspaces of $F$. Let $M$ be the subfield of $F$ which stabilizes $ AB$, i.e. $M=\{x\in F: x\langle AB\rangle \subset \langle AB \rangle\}$. Then 
\begin{align*}
\dim_K \langle AB \rangle \geq\dim_K A+\dim_K B-\dim _K M.
\end{align*}
\end{proposition}

For nonempty subsets $C$ and $D$ of $F$ we have $K\langle C\cup B \rangle=K\langle C \rangle+K\langle D\rangle$, the sum of two subspaces $K\langle C \rangle$ and $K\langle D\rangle$. We have also $K\langle CD \rangle=K(K\langle C\rangle K\langle B\rangle)$.

The following theorem by Eliahou and Lecouvey, which formulates the matching property in terms of suitable dimension estimates, is also the engine behind our proof.

\begin{proposition}\label{dimension estimate}
Let $K\subset F$ be a field extension and $A$ and $B$ be two $n$-dimensional $K$-subspaces of $F$. Suppose that $\mathcal{A}=\{a_1,\ldots,a_n\}$ is a basis of $A$. Then $\mathcal{A}$ can be matched to a basis of $B$ if and only if, for all $J\subset \{1,\ldots,n\}$, we have:
\begin{align*}
\dim_K\bigcap_{i\in J}\left(a_i^{-1}A\cap B\right)\leq n-\# J.
\end{align*}
\end{proposition}

We will also use the following definition which is analogous to the notion of ``coset'' in the group setting. 

\begin{definition}
Let $K\subset F$ be a field extension and $M$ be an intermediate subfield of it. Then a {\it nontrivial linear translate} of $M$ is a $K$-subspace of the form $xM$  for a nonzero element  $x\in F$.
\end{definition}

Now, we are ready to prove Theorem \ref{Linear analog of matchings}. We acknowledge that a similar method has also been suggested in \cite{Aliabadi 3}. It is worth pointing out that the following sufficient condition may be seen as the linear analogue of Proposition \ref{Matching generalziation}. 

\begin{proof}[Proof of Theorem \ref{Linear analog of matchings}.]
Assume to the contrary that  $A$ is not matched to $B$. Then, by Theorem \ref{dimension estimate}, there exists a basis $\mathcal{A}=\{a_1,\ldots,a_n\}$ of $A$ and $J\subset \{1,\ldots,n\}$ such that 
\begin{align*}
\dim_K\bigcap_{i\in J}\left(a^{-1}_iA\cap B\right)> n-\# J.
\end{align*}
Let $S=\langle a_i: i\in J\rangle$ be a $K$-subspace of $A$, $U=\underset{i\in J}{\bigcap}\left(a_i^{-1}A\cap B\right)$ and $U_0=\langle U\cup \{1\}\rangle$. By Proposition \ref{subpaces product}, there exists an intermediate subfield $M$ of $K\subset F$ such that 
\begin{align}
\dim_K\langle U_0S\rangle  \geq \dim_KU_0+\dim_KS-\dim_K M,
\end{align}
where $M$ is the stabilizer of $\langle U_0S\rangle$.  Define $U'=M\cup U$. Invoking Proposition \ref{subpaces product} one more time, one can find an intermediate subfield $M'$ of $K\subset F$ for which
\begin{align}\label{eqn3}
\dim_K\langle U'S\rangle\geq\dim_K\langle U'\rangle+\dim_K S-\dim_K M',
\end{align}
where $M'$ is the stabilizer of $\langle U'S\rangle$. The following computations show that $\langle U'S\rangle=\langle U_0S\rangle$;
\begin{align}\label{eqn4}
\langle U'S\rangle=&\langle(M\cup U)S\rangle=\langle MS\rangle\cup \langle U_0S\rangle\notag\\
=&\langle MS\rangle\cup\langle U_0SM\rangle=M\langle S\cup U_0S\rangle\notag\\
=&M\langle U_0S\rangle=\langle U_0S\rangle.
\end{align}
Then, the stabilizers of these two subspaces must be the same. That is, $M=M'$. Then we would have 
\begin{align}\label{eqn5}
\dim_K\langle U'S\rangle\geq\dim_K\langle U'\rangle+\dim_KS-\dim_K M.
\end{align}
Having \eqref{eqn4} and \eqref{eqn5} at hand and  using the inclusion-exclusion principle for vector spaces we obtain:
\begin{align}\label{eqn6}
\dim_K\langle U_0S\rangle=&\dim_K\langle U'S\rangle\notag\\
\geq& \dim_K\langle U'\rangle+\dim_K S-\dim_K M\notag\\
=&\dim_K\langle M\cup U\rangle+\dim_KS-\dim_KM\notag\\
=&\dim_KM+\dim_KU-\dim_K(M\cap U)+\dim_K S-\dim_KM\notag\\
=&\dim_KU+\dim_KS-\dim_K(M\cap U).
\end{align}
We now have two cases for $M\cap U$:
\begin{enumerate}
\item
If $M\cap U=\{0\}$, $\dim_K(S\cup SU)> n$. On the other hand, since $S\cup SU\subset  A$, we would have $\dim_KA>n$, contradicting the assumption $\dim_KA=n$.
\item
If $M\cap U\neq\{0\}$, then $M\cap B\neq\{0\}$. Choose a nonzero element $b\in M\cap B$. Also let $x$ be a nonzero element of $US$. Then $xK(b)\subset USM\subset A$, contradicting the assumption that $A$ does not contain any nontrivial linear translate of $K(b)$.
\end{enumerate}
Therefore $A$ is matched to $B$, as claimed.
\end{proof}

Back to the group setting, for two finite subsets $A$ and $B$ of a group $G$ with $\#A=\#B>0$, clearly the condition $A\cap (A+B)=\emptyset$ implies that $A$ is matched to $B$. One can go further and argue that every bijection from $A$ to $B$ is a matching. The linear analogue of this statement is studied in \cite[Theorem 6.3]{Eliahou 3}, in which it is proved that for $n$-dimensional $K$-vector spaces $A$ and $B$, the condition $A\cap \langle AB \rangle=\{0\}$ not only implies that $A$ is matched to $B$, but also that every isomorphism from $A$ to $B$  is a strong matching. Another obvious observation in the group setting is that if $A$ is matched to $B$, then $A+B\neq A$. In Theorem \ref{vector space span} we formulate the linear analogue of this observation. To prove Theorem \ref{vector space span} we will need the following lemma the proof of which is obtained along the same lines as used in the proof of Lemma \ref{Coset}.

\begin{lemma}\label{linear translate}
Let $K\subset F$ be a field extension and $A$ and $B$ be finite-dimensional $K$-subspaces of $F$. Assume further that $0<\dim_{K} A\leq \dim_{K} B$ and $\langle AB \rangle=A$. Then $B$ is a subfield of $F$ and $A$ is a linear translate of $B$.
\end{lemma}
\begin{proof}
If $b$ is a nonzero element of $B$, then the linear transformation $T:A\to Ab$ is injective, and thus $\dim_KAb=\dim_KA$. Since $Ab\subset \langle AB \rangle=A$ and $A$ is finite-dimensional, it follows that $Ab=A$. Let  $M=\{x\in L: Ax=A\}$. Then $B\subset M$ and $M$ is a subfield of $L$. Also $AM=A$. If $a\in A$ is nonzero, then $aM$ is a linear translate of the subfield $M$, and thus $\dim_K(aM)=[M:K]$ and we have $aM\subset AM=A$. Then $\dim_KB\geq \dim_KA\geq \dim_K(aM)=[M:K]\geq \dim_KB$, so $\dim_KB=[M:K]$ and $\dim_K A=\dim_K(aM)$. Since $B\subset M$ and $\dim_KB=[M:K]$, it follows that $B=M$, so $B$ is a subfield as claimed.

Since $aM\subset A$, and $\dim_KA=\dim_K(aM)$, it follows that $A=aM$. We know that $M=B$, so $A=aM=aB$, and thus $A$ is the linear translate $aB$ of $B$. The proof is complete.
\end{proof}
\begin{proof}[Proof of Theorem \ref{vector space span}]
Assume to the contrary that $\langle AB \rangle=A$. Then by Lemma \ref{linear translate}, $B$ is a subfield of $F$ and so $1\in B$. Applying Lemma 2.3 in \cite{Eliahou 3} for $A$ and $B$ implies that $A$ cannot be matched to $B$, contradicting our assumption. 
\end{proof}

\subsection{Unmatchable subspaces in a field extension} 
The purpose of this subsection is to formulate the linear analogue of Theorem \ref{Unmatchable} in a field extension $K\subset F$. In the process, we use the dimension estimate of matchable subspaces (Proposition \ref{dimension estimate}) which is derived naturally from the linear version of  Hall's marriage theorem. We assume that $K,F, A,B$ and $n$ are as in Section 3. We assume that $A$ is not matched to $B$. Our goal is to estimate the dimension of the largest subspace of $A$ which is matched to a subspace of $B$, denoted by $M(A,B)$. Since $A$ is not matched to $B$, there exists a basis $\mathcal{A}=\{a_1,\ldots,a_n\}$ which fails the dimension criteria, namely for some $J\subset\{1,\ldots,n\}$,
\begin{align*}
\dim\bigcap_{i\in J}\left(a_i^{-1}A\cap B\right)> n-\# J.
\end{align*}
Define:
\[D_{\mathcal{A}}(B)=\max\{\#J:\ J\subset\{1,\ldots,n\},\;\; \mathrm{and}\;\;\dim\underset{i\in J}{\bigcap}\left(a_i^{-1}A\cap B\right)>n-\#J\},\]
and
\[D(A,B)=\max\{D_{\mathcal{A}}(B):\ \mathcal{A}\; \text{is a basis for}\;A\}.\]
We now formulate and conjecture the linear analogue of Theorem \ref{Unmatchable} as follows. 

\begin{conjecture}
Let $K\subset F$ be a field extension and $A$ and $B$ be two $n$-dimensional $K$-subspaces of $F$ with $n\geq1$, and $\langle AB\rangle \neq A$. Assume that $A$ is not matched to $B$. Then $M(A,B)=n-D(A,B)$. 
\end{conjecture}

Note that according to Theorem \ref{vector space span}, if $\langle AB\rangle=A$, then $M(A,B)=0$.

\section{Matching and Dimension $m$-Intersection Property}\label{intersection property}
The main objective of this section is to present a linear algebra result whose proof relies on tools utilized in matching theory. The first tool which is heavily used in matching theory is called the $m$-intersection property.
The concept of the $m$-intersection property was first studied in \cite{Brualdi} to investigate the sparse basis problem.
(see \cite{Friedland} for more results on the sparse basis problem.)
Following Brualdi, Friedland and Pothen \cite{Brualdi}, we say that the family $\mathcal{J}=\{J_1,\ldots,J_t\}$ of subsets of $\{1,\ldots,n\}$, each of cardinality $m-1$, satisfies the {\it $m$-intersection property} provided that
\begin{align*}
\#\bigcap_{i\in J}J_i\leq m-\# J,
\end{align*}
for any $J\subset\{1,\ldots,t\}$, $J\neq\emptyset$.
It is known that for a given set $\mathcal{J}$ one can check efficiently, i.e. in polynomial time, whether $\mathcal{J}$ satisfies the $m$-intersection property.
 This notion is generalized in \cite{Aliabadi 4} as follows:
\begin{definition}
The family $\mathcal{J}=\{J_1,\ldots,J_t\}$ of subsets of $\{1,\ldots,n\}$, each of cardinality $\leq m-1$, satisfies the {\it weak $m$-intersection property} provided
\begin{align*}
\#\bigcap_{i\in J}J_i\leq m-\# J,
\end{align*}
for all $J\subset \{1,\ldots,t\}$, $J\neq\emptyset$.
\end{definition}
Given an abelian group $G$ and finite subsets $A$ and $B$ of $G$ with $\# A=\# B=n>0$, ``whether $A$ is matched to $B$'' is characterized in \cite{Aliabadi 4} based on ``whether a certain family of subsets of $A$ possesses the weak $n$-intersection property.''

The intersection property introduced above may be of interest in its own right. The following result is proved in \cite{Brualdi}.
\begin{theorem}\label{m-intersection property 1}
Let $J_1,J_2,\ldots,J_t$ be $t<m$ subsets of $\{1,\ldots,n\}$, each of cardinality $m-1$, and assume the $m$-intersection property 
\begin{align}\label{eqq*}
\#\bigcap_{i\in J}J_i\leq m-\# J,
\end{align}
holds for all nonempty subsets $J$ of $\{1,\ldots,t\}$. Then there exist $m-t$ subsets $J_{t+1},\ldots,J_m$ of $\{1,\ldots,n\}$ of cardinality $m-1$ such that \eqref{eqq*} holds for all nonempty subsets $J$ of $\{1,2,\ldots,m\}$.
\end{theorem}

We aim to provide the linear analogue of Theorem \ref{m-intersection property 1}. For this sake, we first define the dimension $m$-intersection property, which is analogous to the notion of the $m$-intersection property in set theory. 
\begin{definition}
Let $W$ be an $n$-dimensional vector space and $U_1,\ldots,U_t$ be $t$ subspaces of $W$ of dimension $m-1$. We say that the family $\mathcal{U}=\{U_1,\ldots,U_t\}$ satisfies the {\it dimension $m$-intersection property} provided that
\begin{align*}
\dim\bigcap_{i\in J}U_i\leq m-\#J,
\end{align*}
for any $J\subseteq\{1,\ldots,t\}$, $J\neq\emptyset$.
\end{definition}
Note that the dimension $m$-intersection property is used to study matchable bases of subspaces in a given field extension. See \cite{Aliabadi 4} for more details. We now formulate the linear analogue of Theorem \ref{m-intersection property 1}, whose proof is obtained from a simple adaption of the proof of Lemma 5.1 in \cite{Brualdi}, in terms of our linear situation.
\begin{observation}\label{m-intersection property 2}
Let $W$ be an $n$-dimensional vector space and $U_1,U_2,\ldots,U_t$ be $t<m$ subspaces of $W$, each of dimension $m-1$, and assume that the dimension $m$-intersection property 
\begin{align}\label{eqq**}
\dim\bigcap_{i\in J}U_i\leq m-\# J,
\end{align}
holds for all nonempty subsets $J$ of $\{1,\ldots,t\}$. Then there exist $m-t$ subspaces $U_{t+1},\ldots,U_m$ of $W$ of dimension $m-1$ such that \eqref{eqq**} holds for all nonempty subsets $J$ of $\{1,2,\ldots,m\}$.
\end{observation}
\begin{proof}
It suffices to change ``subsets'', ``cardinality'', and ``$m$-intersection property'' to ``subspaces'', ``dimension'', and ``dimension $m$-intersection property'', respectively, in the proof of Lemma 5.1 in \cite{Brualdi}. Then the same argument along with the inclusion-exclusion principle for vector spaces shall complete the proof.
\end{proof}
\subsection{Linear analogue of Hall's marriage theorem}
The second tool which is used in the proof Theorem \ref{linear algebra} is a linear analogue of Hall's marriage theorem. Let $W$ be a vector space over a field $K$ and $\mathcal{U}=\{U_1,\ldots,U_m\}$ be a family of $K$-subspaces of $W$. A {\it free transversal} for $\mathcal{U}$ is a set of  linearly independent vectors $\{x_1,\ldots,x_m\}$ in $W$ with $x_i\in U_i$, $1\leq i\leq m$. In the following theorem due to Rado \cite{Rado} the  necessary and sufficient conditions for the existence of a free transversal of $\mathcal{U}$ are given.
\begin{theorem}\label{transversal}
Let $W$, $K$ and $\mathcal{U}$ be as above. Then $\mathcal{U}$ has a free transversal if and only if 
\begin{align*}
\dim\underset{i\in J}{+} U_i\geq\#J,
\end{align*}
for all $J\subseteq \{1,\ldots,m\}$.
\end{theorem}
\subsection*{A Few Notations}
We shall use the following standard notation. We denote 
\begin{align*}
W^*=\left\{\psi:W\to K:\, \psi \text{ is linear transformation}\right\},
\end{align*}
the {\it dual} of $W$. Moreover, for any subspace $V$ of $W$, we denote
\begin{align*}
V^\perp=\left\{\psi\in W^*:\, V\subset \ker\psi\right\},
\end{align*}
the {\it orthogonal} of $V$ in $W^*$. We will also use the fact that $V\oplus V^\perp=W$. 

Having Observation \ref{m-intersection property 2} and Theorem \ref{transversal} at hand we are ready to prove Theorem \ref{linear algebra}. 

\begin{proof}[Proof of Theorem \ref{linear algebra}.]
According to Observation \ref{m-intersection property 2}, one may find $(n-1)$-dimensional subspaces $U_{t+1},\ldots,U_n$ such that the family $\{U_1,\ldots,U_n\}$ satisfies the $n$-intersection property. Hence, for any $J\subseteq\{1,\ldots,n\}$, $J\neq\emptyset$, we have 
\begin{align}\label{eqq(..)}
\dim\bigcap_{i\in J}U_i\leq n-\# J.
\end{align}
Taking the orthogonal in the dual space $W^*$, we have 
\begin{align*}
\dim\left(\bigcap_{i\in J}U_i\right)^\perp \geq \#J.
\end{align*}
Thus,
\begin{align*}
\dim \sum_{i\in J}U_i^\perp \geq \# J.
\end{align*}
It follows from Theorem \ref{transversal} that there exists a free transversal $(\psi_1,\ldots,\psi_n)\in W^*$ for the family of subspaces $\{U_i^\perp\}_{i=1}^n$. Since $\psi_1,\ldots,\psi_n$ are linearly independent and $\dim W^*=n$, then $\{\psi_i\}_{i=1}^n$ forms a basis for $W^*$. Let $\{x_1,\ldots,x_n\}$ be a basis for $W$ for which $x_i^*=\psi_i$, $1\leq i\leq n$. Then $U_i\subseteq \ker x_i^*$. Condition \eqref{eqq(..)} implies that $U_i\neq U_j$, for some $1\leq i, j\leq n$ and since $\dim U_i=\dim U_j=n-1$, then $U_i+U_j=W$. We totally have $W= U_i+U_j\subseteq \ker x_i^*+\ker x_j^*\subseteq W $, which implies the desired result.
\end{proof}

\section{Future work}\label{Future}
 The matching problems considered in this paper can be reformulated for matroids in a field extension in some seemingly unchallenging ways; let $K\subset F$ be a field extension, $A$ be a subset of $F$, and $M_1$ and $M_2$ be two matroids over $K(A)$, where $K(A)$ stands for the subfield of $F$ generated by $A$. Then adapting the notion of matchable subspaces, one may define matchable matroids. The main obstacle in finding the matroid analogue of matchings may be the matroid version of Proposition \ref{subpaces product}. We hope that the techniques presented in \cite{Bachoc} have more general
applicability, especially in the direction of generalizing these statements to matroids in a field extension.
 
\section*{Acknowledgement}
We are deeply grateful to Shira Zerbib and Khashayar Filom for their constant encouragement, generosity, and for many insightful conversations. This work was supported by the Iowa State University Dean’s High Impact Award for undergraduate summer research in mathematics.

\end{document}